\documentclass[12pt]{amsart}
\usepackage{amsmath,amsfonts,amssymb,amsthm,epsfig,color, tikz}
\usepackage{mathrsfs}
\newcommand{\bb}{\mathbb}

\newcommand{\C}{\bb C}

\newcommand{\Z}{\bb Z}
\newcommand{\R}{\bb R}

\newcommand{\eps}{\epsilon}

\newcommand{\cH}{\mathcal{H}}

\newcommand{\cO}{\mathcal{O}}

\newcommand{\cS}{\mathcal{S}}

\newcommand{\bS}{\mathbb{S}}
\newcommand{\bR}{\mathbb{R}}

\newcommand{\onto}{\xymatrix{\ar@{>>}[r]&}}
\newcommand{\da}[4]{\xymatrix{#1 \ar@<.5ex>[r]^{#2} \ar@<-.5ex>[r]_{#3} & #4}}

\newif\ifdraft\drafttrue

\newcommand{\SLR}[1][2]{\operatorname{SL}_{#1}(\bR)}

\newcommand\on[1]{{\operatorname{#1}}} 
\newcommand\diag[1]{\operatorname{diag}\left(#1\right)}

\newcommand{\acts}{\hspace{-1pt}\mbox{\raisebox{1.3pt}{\text{\huge{.}}}}\hspace{-1pt}}

\newcommand{\id}{\mathbbm{1}}

\newcommand{\Lie}{\on{Lie}}
\usepackage{bbm}

\newcommand{\prim}{\operatorname{prim}}

\newcommand{\om}{\omega}

\newcommand{\hh}{\mathcal H}

\newcommand{\minuszero}{\backslash\{0\}}
\newcommand{\La}{\Lambda}
\newtheorem{Theorem}{Theorem}
\numberwithin{Theorem}{section}
\newtheorem{Cor}[Theorem]{Corollary}

\newtheorem{lemma}[Theorem]{Lemma}

\newtheorem*{question*}{Question}
\makeatletter\let\@wraptoccontribs\wraptoccontribs\makeatother
\newtheorem*{lemma*}{Lemma}

\newtheorem*{theorem*}{Theorem}

\numberwithin{equation}{section}
\numberwithin{Def}{section}

\theoremstyle{remark}
\newtheorem{remark}{Remark}

\begin{document}
\title{Siegel-Veech transforms are in $L^2$}
\author{Jayadev S.~Athreya}
\author{Yitwah Cheung}
\author{Howard Masur}

\email{jathreya@uw.edu}
\email{ycheung@sfsu.edu}
\email{masur@math.uchicago.edu}
\address{Department of Mathematics, University of Washington, Padelford Hall, Seattle, WA 98195, USA}
\address{Department of Mathematics, San Francisco State University, Thornton Hall 937, 1600 Holloway Ave, San Francisco, CA 94132, USA}
\address{Department of Mathematics, University of Chicago, 5734 South University Avenue, Chicago, IL 60615, USA}
    \thanks{J.S.A. partially supported by NSF CAREER grant DMS 1559860}
    \thanks{Y.C. partially supported by NSF DMS 1600476}
    \thanks{H.M. partially supported by  NSF DMS 1607512}
    \contrib[with an appendix by]{Jayadev S.~Athreya and Rene R\"uhr}
\address{Department of Mathematics 
University of Toronto 
Room 6290, 40 St. George Street 
Toronto, ON, M5S 2E4 Canada.
}
\email{anja@math.toronto.edu}
\maketitle

\begin{center}
\textit{Dedicated to the memory of William Veech.}
\end{center}

\begin{abstract} Let $\hh$ denote a connected component of a stratum of translation surfaces. We show that the Siegel-Veech transform of a bounded compactly supported function on $\R^2$ is in $L^2(\hh, \mu)$, where $\mu$ is Lebesgue measure on $\hh$, and give applications to bounding error terms for counting problems for saddle connections. We also propose a new invariant associated to $SL(2, \R)$-invariant measures on strata satisfying certain integrability conditions.
\end{abstract}

\section{Introduction}\label{sec:intro} Motivated by counting problems for polygonal billiards and more generally for linear flows on surfaces, Veech~\cite{VeechSiegel} introduced what is now known as the \emph{Siegel-Veech transform} on the moduli space of abelian differentials (in analogy with the Siegel transform arising from the space of unimodular lattices in $\R^n$). The main result of~\cite{VeechSiegel} is an integration ($L^1$) formula for this transform (see \S\ref{sec:SVformula}), a version of the classical Siegel integral formula. 

Our main result Theorem~\ref{theorem:main} is that the Siegel-Veech transform $\widehat{f}$ of any bounded compactly supported function $f$ on $\R^2$ satisfies $$\widehat{f} \in L^2(\hh, \mu)$$ with respect to the natural Lebesgue measure $\mu$ on any (connected component of a) stratum $\hh$ of abelian differentials.  It seems an interesting question what the closure of the set of $\widehat{f}$ is in $L^2(\hh, \mu)$.  

\subsection{Translation surfaces}\label{sec:translation} A \emph{translation surface} $S$ is a pair $(X, \omega)$ where $X$ is a Riemann surface and $\omega$ is a holomorphic $1$-form. The terminology is motivated by the fact that integrating $\omega$ (away from its zeros) gives an atlas of charts to $\C$ whose transition maps are translations. These can be viewed as singular flat metrics with trivial rotational holonomy, with isolated cone-type singularities corresponding to zeros of $\omega$.  An \emph{saddle connection} $\gamma$ on $S$ is a geodesic segment connecting two zeros of $\omega$ with none in its interior. Associated to each saddle connection is its \emph{holonomy vector} $$z_{\gamma} = \int_{\gamma} \omega \in \C$$ and its length
$$|\gamma|=\int_{\gamma}|\omega|.$$  We denote the set of holonomy vectors by $\La_{\omega}$. $\La_{\om}$ is a discrete subset of the plane $\C \sim \R^2$.

\subsection{Strata}\label{sec:strata} The moduli space $\Omega_g$ of genus $g$ translation surfaces is the bundle over the moduli space $\mathcal M_g$ of genus $g$ Riemann surfaces with fiber over each Riemann surface $X$ given by $\Omega(X)$, the vector space of holomorphic 1-forms on $X$. $\Omega_g$ decomposes into \emph{strata} depending on the combinatorics of the differentials. 

Since the orders of the zeros of $\omega$ must sum to $2g-2$, there is a stratum associated to each integer partition of $2g-2$. Each of these strata has at most three connected components~\cite{KZ}. 

The flat metric associated to  a one-form $\omega$ also gives a notion of area on the surface.  We consider the subset of \emph{area 1} surfaces of a connected component of a stratum, and denote it by $\hh$. We will, by abuse of notation, often simply refer to this as a stratum, and we will denote elements of it by $(X,\omega)$.

\subsubsection{Lebesgue measure}\label{sec:lebesgue} The group $GL(2, \R)$ acts on $\Omega_g$ via linear post-composition with charts, preserving combinatorics of differentials. The subgroup $SL(2,\R)$ preserves each area 1 subset, so acts on each stratum $\hh$. On each stratum, there is a natural measure $\mu$, known as \emph{Masur-Veech} or Lebesgue measure, constructed using \emph{period coordinates} on strata (see, e.g.~\cite{Eskinhandbook} or~\cite{Zorichsurvey} for a nice exposition of the construction of this measure). A crucial result, independently shown by W.~Veech and the H.~Masur, is

\begin{theorem*}\cite{MasurErgodic, VeechErgodic} $\mu$ is a finite $SL(2,\R)$-invariant ergodic measure on each stratum $\hh$.

\end{theorem*}

\subsection{Siegel-Veech transforms}\label{sec:SV} Fix a stratum $\hh$. Let $B_c(\R^2)$ denote the space of Borel measurable bounded compactly supported functions on $\R^2 \cong \C$. Given $(X,\omega) \in \hh$ and $f \in B_c(\R^2)$, Veech~\cite{VeechSiegel} introduced the \emph{Siegel-Veech transform} $$\widehat{f}(X,\omega) = \sum_{v \in \La_{\omega}} f(v).$$ Note that this is a \emph{finite} sum for any fixed $f$ and $\omega$, since $\La_{\om}$ is discrete. Our main result is:

\begin{Theorem}\label{theorem:main} Let $f \in B_c(\R^2)$. Then $\widehat{f} \in L^2(\hh, \mu)$.

\end{Theorem}

\medskip
\noindent This corrects a mistake in~\cite{AthreyaChaika}, which claimed that if $f$ was the indicator function of the unit disk, that $\widehat{f} \notin L^2$. In fact, the proof in~\cite{AthreyaChaika} only shows $\widehat{f} \notin L^3$.



\subsubsection{Siegel-Veech formulas}\label{sec:SVformula} Veech~\cite{VeechSiegel} showed $\widehat{f} \in L^1(\hh, \mu)$, and using the $SL(2, \R)$-invariance of $\mu$ and a classification of the $SL(2, \R)$-invariant measures on $\R^2$, concluded 

\begin{theorem*}\cite{VeechSiegel} There is a constant $c = c(\mu)$ so that $$\int_{\hh} \widehat{f} d\mu = c \int_{\R^2} f dm,$$ where $m$ is Lebesgue measure on $\R^2$.

\end{theorem*}

This is a generalization of the \emph{Siegel integral formula} \cite{Siegel}, which applies to averages of similar transforms over spaces of unimodular lattices.

\subsubsection{Siegel-Veech constants}\label{sec:SVconstants} In fact, Veech showed (using Masur's quadratic upper bounds~\cite{MasurCounting}) that for \emph{any} $SL(2,\R)$-invariant ergodic finite measure $\lambda$, that $\widehat{f} \in L^1(\hh, \lambda)$, so there is $c = c_{SV}(\lambda)$ so that $$\int_{\hh} \widehat{f} d\lambda = c \int_{\R^2} f dm.$$ He called these measures \emph{Siegel measures} on $\hh$. These constants $c_{SV}(\lambda)$ are known as \emph{Siegel-Veech constants} and are important numerical invariants associated to $SL(2, \R)$-invariant measures.

\subsection{Siegel-Veech measures}\label{sec:SVmeasures}  For any measure $\lambda$  
with $\widehat{f} \in L^2(\hh, \lambda)$, we can define two measure-valued invariants. First, we extend the notion of Siegel-Veech transform to $B_c(\R^4)$, viewing $\R^4 = \R^2 \times \R^2$. 

\subsubsection{Generalized Siegel-Veech transforms} Given $$h \in B_c(\R^4) = B_c(\R^2 \times \R^2).$$ define the \emph{Siegel-Veech transform} $$\widehat{h}(\omega) = \sum_{v_1, v_2 \in \Lambda_{\omega}} h(v_1, v_2).$$ Note that if $h(x,y) = f(x)f(y)$ for $f \in B_c(\R^2)$, $$\widehat{h} = (\widehat{f})^2.$$ 
\subsubsection{Measure-valued invariants}
\begin{Theorem}\label{theorem:spectral} Let $\lambda$ denote an $SL(2,\R)$-invariant measure on $\hh$ so that for any $f \in B_c(\R^2)$, $\widehat{f} \in L^2(\hh, \lambda)$. Let $\kappa$ denote normalized Haar measure on $SL(2,\R)$.  Then there exist \emph{Siegel-Veech measures} $\nu = \nu(\lambda)$ on $\R\minuszero$ and $\eta = \eta(\lambda)$ on $\mathbb P^1 (\R) =  \R \cup \infty$ such that for any $h \in B_c(\R^4)$, \begin{eqnarray*} \int_{\mathcal H} \widehat{h}(\omega) d\lambda(\omega) &=& \int_{\R\minuszero} \left(\int_{SL(2, \R)} h(tx, y) d\kappa(x,y) \right) d\nu(t) \\ &+ &\int_{\mathbb P^1(\R)} \left(\int_{\R^2} h(x, sx) dx\right) d\eta(s).\\ \end{eqnarray*}
\end{Theorem}

\medskip

\noindent 

\subsection{Counting bounds}\label{sec:counting} Veech introduced the Siegel-Veech transform to understand counting problems. Given a translation surface $(X,\omega)$ (or simply  $\omega$),  let $$N(\omega, R) = \#\left(\La_{\om} \cap B(0, R)\right)$$ denote the number of saddle connections of length at most $R$. Masur~\cite{MasurCounting} showed that there are constants $$0 < c_1 = c_1(\omega) \le c_2 = c_2(\omega)$$ so that $$c_1 R^2 \le N(\omega, R) \le c_2 R^2.$$ Dozier~\cite{Dozier} recently improved these bounds, showing that there is a \emph{uniform} $c$ and $R(\omega)$ so that $$N(\omega, R) \le c R^2 \mbox{ for } R > R(\omega).$$  The Siegel-Veech formula computes the \emph{mean} of $N(\omega, R)$,  $$\int_{\hh} N(\omega, R) d\mu(\om) = c_{SV}(\mu) \pi R^2.$$ Eskin-Masur~\cite{EskinMasur} showed that for $\mu$-almost every $\omega \in \hh$, $$\lim_{R \rightarrow \infty} \frac{N(\omega, R)}{\pi R^2} = c_{SV}(\mu).$$

\subsubsection{Error terms}\label{sec:error} Our results can be viewed as showing that Siegel-Veech transforms have finite \emph{variance}. Variance bounds in turn yield \emph{concentration bounds}, bounding the probability of large discrepancy from the mean. Finally, combining concentration bounds with the Borel-Cantelli lemma yield almost everywhere error term bounds for $N(\omega, R)$. Suppose we write $$L(R) = \|N(\omega, R)\|_2^2,$$ and let $e(R)$ denote an \emph{error} function. Define $$V(R) = L(R) - \|N(\omega, R)\|_1^2 = L(R) - c_{SV}(\mu)^2 \pi^2 R^4.$$

\begin{Theorem}\label{theorem:error} Let $R_k \rightarrow \infty$ be a sequence, and $e$ be a function such that $$\sum_{k=1}^{\infty} \frac{V(R_k)}{e(R_k)^2} < \infty.$$Then for $\mu$-almost every $\om \in \hh$, there is a $k_0$ so that for all $k \geq k_0$ $$\left|N(\omega, R_k) - c_{SV}(\mu) \pi R_k^2\right| < e(R_k).$$

\end{Theorem}

\medskip

\noindent The main result Theorem~\ref{theorem:variance} of Appendix~\ref{sec:appendix} gives a power savings bound for $|L(R)- c_{\mu} \pi R^2|$. Using this, we obtain in \S\ref{sec:powersavings} almost everywhere power savings along lacunary sequences. Recently, Nevo-R\"uhr-Weiss~\cite{NRW}, using exponential mixing of Teichm\"uller flow, give error bounds with a power savings (along the full sequence).

\subsection{Organization of the paper}\label{sec:organization} In \S\ref{sec:delaunay}, we describe technical result about Delaunay triangulations that is crucial in our proofs. In \S\ref{sec:measurebounds}, we prove Theorem~\ref{theorem:main} via intermediate results Theorem~\ref{theorem:main:1} and Theorem~\ref{theorem:main:2}, which prove the result for the indicator function of the unit disk and of a ball of radius $R$ respectively.  In \S\ref{sec:spectral}, we prove Theorem~\ref{theorem:spectral} and discuss explicit computations in special cases. 

\subsubsection{Acknowledgments} We would like to thank Alex Eskin, Duc-Manh Nguyen, Kasra Rafi, Rene Ruhr, John Smillie,  and Barak Weiss,  for useful discussions. Both Ben Dozier and the anonymous referee made suggestions which greatly improved the paper. We dedicate this paper to the memory of William Veech.  This project was initiated during the Spring 2015 programs at the Mathematical Sciences Research Institute (MSRI), ``Geometric and Arithmetic Aspects of Homogeneous Dynamics" and ``Dynamics on moduli spaces of geometric structures". This work was continued at the No Boundaries conference at the University of Chicago in Fall 2017; the Mathematisches Forschungsinstitut Oberwolfach (MFO) workshop on ``Flat Surfaces and Algebraic Curves" in Fall 2018; and completed at the Fields Institute program on ``Teichmüller Theory and its Connections to Geometry, Topology and Dynamics" in Fall 2018. We thank the organizers of these meetings, MSRI, the University of Chicago, MFO, and the Fields Institute for their hospitality and support.


\section{Delaunay triangulations and Chew's theorem}\label{sec:delaunay} In this section, we review the concepts of a Delaunay triangulation (introduced in~\cite{MasurSmillie}) of a translation surface and show how to adapt a theorem of Chew~\cite{Chew} from the Euclidean plane to the setting of translation surfaces.

\subsection{Delaunay triangluations}\label{subsec:delaunay} Let $S = (X, \omega)$ be a translation surface. Pulling back the $L^1$ metric from $\C$ using the charts determined by $\omega$, we can consider the \emph{Voronoi decomposition} of the translation surface with respect to this metric: each singular point determines a cell given by points which are closer to it than to any other singular point, and have a unique shortest geodesic connecting the two. The dual decomposition is the Delaunay decomposition of the surface, and any triangulation given by a further dissection of this is a $L^1$ \emph{Delaunay triangulation} of $(X, \omega)$. We recall the following elementary but important lemma:

\begin{lemma}\label{lemma:shortest} Let $\gamma$ be the shortest saddle connection on $S$. Then $\gamma$ is an edge in any Delaunay triangulation of $S$. 
\end{lemma}
\begin{proof} Consider the endpoints (possibly the same point) of $\gamma$. If the points are distinct, the edge between the two Voronoi regions containing the points is a segment of the perpendicular bisector of $\gamma$. Thus $\gamma$ is an edge of the Delaunay decomposition. If $\gamma$ is closed, the perpendicular bisector of $\gamma$ consists of points where there are two shortest paths to the singularity, and again is an edge of the Voronoi decomposition. 
\end{proof}

\subsection{Chew's algorithm}\label{subsec:chew} We will use the Delaunay triangulation to build paths that approximate saddle connections. The ideas come from the remarkable work in computational geometry by Chew \cite{Chew}. He proved the following:

\begin{Theorem}\label{theorem:chew} Let $\Sigma$ be a set of points in the plane $\R^2$ and let $T$ be an $L^1$ Delaunay triangulation of $\Sigma$. For any points $A$ and $B$ of $\Sigma$, there exists an $A$-to-$B$ path along edges of $T$ that has length at most $\sqrt{10}|AB|$, where $|AB|$ is the standard Euclidean distance between $A$ and $B$.
\end{Theorem}

We will apply Chew's argument (which goes through word-for-word) to the collection of singular points on a translation surface, using the $L^1$ Delaunay triangulation defined above. Precisely, Chew's argument yields the following:

\begin{lemma}\label{lemma:chew} Let $S = (X, \omega)$ be a translation surface, and let $T$ be the $L^1$ Delaunay triangulation with respect to the set of singular points $\Sigma$. Then every saddle connection $\beta$ is homotopically equivalent to a path $P(\beta)$ in the $L^1$-Delaunay triangulation whose total (euclidean) length sastifies  $$ |P(\beta)| \le  \sqrt{10}|\beta|.$$
\end{lemma}

\begin{proof} We describe Chew's algorithm to produce a Delaunay path. We assume, without loss of generality, that the slope of $\beta$ is bounded in absolute value by $1$.  Consider the collection of triangles  in the universal cover of $S$ given by lifting the triangulation. Suppose $\beta$ crosses $N$ triangles in the triangulation. Then the sequence of these triangles crossed by any lift of $\beta$ in the universal cover forms an $(N+2)$-gon.  

Write the triangles from left to right as $\Delta_j, j=1,\ldots, N$, and let $z_0$ be left endpoint of $\beta$. The algorithm recursively selects a subsequence $\Delta_{j_i}$, $i=0,1,2, \ldots$ starting with $j_0=1$ and a sequence of vertices $z_i$ of $\Delta_{j_i}$.  
Given $\Delta_{j_i}$ having $z_i$ as a vertex, let $$j_{i+1} = \max\{j: z_i \mbox{ is a vertex of } \Delta_j\}.$$ $z_{i+1}$ is the vertex of $\Delta_{j_{i+1}}$ chosen as follows: Let $D_i$ denote the \emph{circumscribing diamond} of $\Delta_{j_i}$, that is the domain of an immersed $L^1$ disc (that is, a diamond)
having the vertices of $\Delta_{j_i}$ on its boundary.  We assume our surface has the property that each Delaunay triangle has at most one 
singularity on each side of $D_i$, since the collection of such surfaces is of full measure. Assume $z_i$ is on or above $\beta.$  (The other case is similar.) We divide into two cases:

\begin{enumerate}

\item $z_i$ lies on an upper side of $D_i$.  We choose $z_{i+1}$ to be the first vertex encountered moving clockwise around boundary of $D_i$. Note that $z_{i+1}$ remains above $\beta$.

\medskip

\item $z_i$ lies on the lower side of $D_i$. Note that this implies $z_i$ is on the lower \emph{left} of $D_i$ by the slope assumption.\footnote{In particular, the fourth and last case of Chew's algorithm in \cite{Chew} never occurs.}  We choose $z_{i+1}$ to be the vertex in the lower right. Now $z_{i+1}$ is below $\beta$

\end{enumerate}

Note that in the second case, the path along boundary of $D_i$ from $z_i$ to $z_{i+1}$ 
is a 2-segment v-shape.  The first case has 3 subcases: 

\begin{enumerate}
\item $z_i$ on upper left, $z_{i+1}$ on upper right, the path connecting them on the boundary of $D_i$ is carat (\verb|^|) shaped, consisting of two line segments.

\item $z_i$ on upper left, $z_{i+1}$ on lower right,the path connecting them on the boundary of $D_i$ is \verb|]| shaped, consisting of three line segments (not necessarily at right angles).

\item $z_i$ on upper right, $z_{i+1}$ on lower right, the path connecting them on the boundary of $D_i$ is \verb|>| shaped, consisting of two line segments.
\end{enumerate}

In the last 2 subscases (of the first case) the path along boundary $D_i$ from $z_i$ to $z_{i+1}$ backtracks in the sense that $dx/dt$ is negative on the last segment.  The estimate on length follows because we can only have a limited amount of back-tracking in the sense that if the last segment is extended to hit $\beta$, the point where it hits is to the right of the vertical line of symmetry of $D_i.$
\end{proof}

\begin{remark}\label{remark:chew}
Notice that the saddle connection joining $z_{i-1}$ to $z_i$ makes an angle at $z_i$ strictly less than $2\pi$ with  the saddle connection joining $z_i$ and $z_{i+1}$. This follows from the  fact that each makes an angle strictly smaller than $\pi$ with the last saddle connection crossed by $\beta$.
\end{remark}
\section{Measure bounds and decompositions}\label{sec:measurebounds} In this section we prove Theorem~\ref{theorem:main},  first proving it for the indicator function of a small  disk (Theorem~\ref{theorem:main:1}) and then for the disk of radius $R$ (Theorem~\ref{theorem:main:2}). 

\subsection{Tail bounds}
Given $f \in B_c(\R^2)$,  to prove $\widehat{f} \in L^2(\hh,\mu)$, we need to show \begin{equation}
\label{eq:upper:bound}
\sum_{k =1}^{\infty} \mu\{(X,\omega): \widehat{f}(X,\omega) > \sqrt{k}\}<\infty.
\end{equation}

\subsection{A fixed disc} The first iteration of our main result is:
\begin{Theorem}\label{theorem:main:1}
Fix a small $\epsilon_0$ and let $f:\R^2\to \R$ be the indicator function of the disc of radius $\epsilon_0$.
Then $$\widehat f\in L^2(\hh, \mu).$$
\end{Theorem}

 \begin{proof}
We will divide the stratum $\hh$  into a finite number of subsets $\hh_i$  and prove $$\int_{\hh_i}(\widehat f)^2d\mu <\infty$$ for each $\hh_i$.  

\subsection{Thick part} Let $\hh_1$ be the subset where every saddle connection has length at least $\epsilon_0$. By definition for $(X, \omega) \in \hh_1$, $\widehat f(X,\omega)=0$.

\subsection{No short loops}\label{sec:noshortloops}
  
Let $\hh_2$ be the set of $(X,\omega)$  for which there are saddle connections of length smaller than $\epsilon_0$ but no homotopically nontrivial closed curves of length less than $\epsilon_0$.

Take the $L^1$ Delaunay triangulation of $(X,\omega)$.  There are  no loops with length shorter than $\epsilon_0$.   By Lemma~\ref{lemma:chew} 
any saddle connection is homotopic to a path in the edges of the Delaunay triangulation of length at most a  fixed multiple of the length of the saddle connection. Since there are no loops shorter than $\epsilon_0$, any such path in the edges traverses  successively at most a fixed number of edges shorter than $\epsilon_0$ before traversing one of length at least $\epsilon_0$. Thus  a saddle  connection of  length at most $1$     can be written as a union of at most $O(1/\epsilon_0)$ edges of the triangulation and therefore expressed in terms of a fixed basis for $H_1(X,\omega,\Sigma)$ as a linear combination with coefficients that are $O(1/\epsilon_0)$. Thus there are $O(1/\epsilon_0^N)$ saddle connections, where $N$ is the dimension of $H_1(X,\omega,\Sigma)$.
Since $\epsilon_0$ is fixed, $\widehat{f}$ is bounded on $\hh_2$, that is $\widehat f \in L^{\infty}(\hh_2, \mu) \subset L^2(\hh_2, \mu)$, so $$\int_{\hh_2} (\widehat f)^2 d\mu<\infty.$$

\subsection{Short loops}\label{sec:shortloops}

Now we treat the case that there are short loops of length smaller than some fixed 
$\epsilon_0$.  Let $\hh_3$ be the set of $(X,\omega)$ with a short curve $\gamma$ of length at most $\epsilon_0$. 
Let $N$ be the dimension of relative homology, and choose $$0 < \delta < p < \frac 1 N.$$

\noindent Let $|\gamma|$ be the length of shortest saddle connection (possibly loop)  on $(X,\omega)$, and recall  that $\widehat f(X,\omega)$  counts the number of saddle connections whose holonomy vector  lies in a disc of radius $\epsilon_0$.    By Theorem 5.1 of \cite{EskinMasur}, 
 for some fixed $\epsilon_0$, there is $C$ (depending on $\delta$ but not $|\gamma|$)  such that  \begin{equation}\label{eq:eskinmasur}\widehat f(X,\omega)\leq \frac{C}{|\gamma|^{1+\delta}}.\end{equation}
If $\widehat f(X,\omega)\geq  \sqrt k$, then the above bound says that for $c = C^{\frac{1}{1+\delta}},$ \begin{equation}\label{eq:eskinmasur2}|\gamma|\leq ck^{-\frac{1}{2(1+\delta)}}.\end{equation}

\noindent We will make crucial use of the following:

\begin{lemma}\label{lemma:5crossing}
Suppose $\gamma$ is the shortest saddle connection on $(X, \omega)$. Let  $\beta$ be a saddle connection (with an orientation)such that the path $P(\beta)$ follows edges parallel to $\gamma$ more than $2M+1$ times, where $M$ is the total number of triangles in the Delaunay triangulation.  Then there is a cylinder $C$ with $\gamma$ on its boundary and $\beta$ crosses $C$.

\end{lemma}

\begin{proof}
The saddle connection $\gamma$ appears as an edge in the Delaunay triangulation since it is the shortest saddle connection. Chew's algorithm produces a path $P(\beta)$ that follows the 1-skeleton of the complex of triangles crossed by $\beta$. Lift this path to the universal cover $\widetilde{(X, \omega)}$ of $(X, \omega)$. Let $\Delta_1, \Delta_2, \ldots, \Delta_j$ be the sequence of triangles in $\widetilde{(X, \omega)}$ with an edge parallel to $\gamma$ written in the order they appear in $P(\beta)$. Consider the connected components of the lifts of this  set of edges in $\Delta_i$ parallel to $\gamma$. Since $\beta$ is a geodesic, its lift to $\widetilde{(X, \omega)}$ can cross at most one edge in a component; otherwise in $\widetilde{(X,\omega)}$ the lift of $\beta$ and  the edges in a component would bound a disc, contradicting that each is a geodesic.  Since by assumption, $P(\beta)$ follows at least $2M+2$ edges parallel to $\gamma$,  without loss of generality we can assume that the lift of $\beta$ crosses at least $M+1$ of the triangles $\Delta_1,\dots,\Delta_j$ which have edges parallel to $\gamma$, does not cross any of the edges of these $\Delta_i$ which are parallel to $\gamma$, and these edges are to the right of the lift of $\beta$. 

Since there are only $M$ distinct triangles of $(X,\omega)$, and the lift of $\beta$ crosses $M+1$ triangles, 
there are $\Delta_{i_1}, \Delta_{i_2}$ which are lifts of the same triangle on $(X, \omega)$ such that the lift of  $\beta$ does not cross the edges  $e_{i_1}$, $e_{i_2}$ of $\Delta_{i_1}, \Delta_{i_2}$ which are parallel to $\gamma$.  The angle on the left side between consecutive parallel segments of $P(\beta)$ at a singularity is an integer multiple of $\pi$; by Remark~\ref{remark:chew} it is exactly $\pi$. The corresponding edges of $\Delta_{i_1}$ and $\Delta_{i_2}$ adjacent to $e_{i_1}$ and $e_{i_2}$ are  parallel of the same length.  Subsegments of these together with the saddle connections parallel to $\gamma$ and a segment parallel to $\gamma$ joining the regular endpoints of the subsegments form a  parallelogram on $\widetilde{(X, \omega)}$ which projects to a cylinder on $(X, \omega)$.  Further $\beta$ crosses the corresponding maximal cylinder. 
\end{proof}

\subsection{Decomposing $\hh_3$}\label{sec:decomposition} We break the set of $(X,\omega)\in \hh_3$ such that $\widehat f\geq \sqrt k$ into three sets $\Omega_0(k)\cup\Omega_1(k)\cup\Omega_2(k)$. It suffices to prove $$\sum_k \mu(\Omega_i(k))<\infty$$ for each $i$. Let $\epsilon(X,\omega)$ be the length of the second shortest nonhomlogous saddle connection  on $(X,\omega)$.

\subsubsection{The second shortest nonhomologous saddle connection is shorter than a power of the shortest.}\label{sec:omega0} Let $$\Omega_0(k)=\{(X,\omega)\in \hh_3:\widehat f(X,\omega) \geq \sqrt k\ \text{and}\ \epsilon(X,\omega) \leq 
|\gamma|^p\}.$$ 

\noindent Since we have a saddle connection of length $|\gamma|$ and one of length at most $|\gamma|^p$ we have, by (\ref{eq:eskinmasur2}) 
$$\mu(\Omega_0(k))=O(|\gamma|^{2+2p})=O(k^{-\frac{1+p}{1+\delta}})$$ which is summable since $\delta < p$. Thus $$\sum_k \mu(\Omega_0(k))<\infty.$$

\subsubsection{The second shortest saddle connection is longer than a power of the shortest, and the shortest is not a cylinder curve.}\label{sec:omega1} Let $\Omega_1(k)$ be the set of surfaces $(X,\omega)$ such that the shortest saddle connection $\gamma$ is not on the boundary of a cylinder and $$\widehat {f}(X,\omega)\geq \sqrt{k},\ \epsilon(X,\omega)> |\gamma|^p.$$
Since $\gamma$ is shortest, the saddle connection(s) making up $\gamma$ are edges of the Delaunay triangulation. Then by Lemma~\ref{lemma:5crossing} the path $\beta$ when written as a path in the Delaunay triangulation may follow $\gamma$ successively at most $2M$ times, but then must follow some other edge. We are then reduced to the argument in \S\ref{sec:noshortloops} to give
$$\widehat f(X,\omega) =O\left(\epsilon^{-N}\right)=O\left(|\gamma|^{-Np}\right).$$
For this to be bigger than $\sqrt k$ have $$|\gamma|=O\left(k^{-\frac{1}{2Np}}\right).$$ Thus 
$$\mu(\Omega_1(k))= 
O(|\gamma|^2)=O\left(k^{-\frac{1}{Np}}\right)$$ which is summable since $Np < 1$ and again we have $$\sum_k\mu(\Omega_1(k))<\infty.$$

\subsubsection{The second shortest saddle connection is longer than a power of the shortest, and the shortest is a cylinder curve.}\label{sec:omega2} Let $\Omega_2(k)$ be the set of $(X,\omega)$ such that 
$$\widehat f(X,\omega)\geq \sqrt k,\ \epsilon(X,\omega)> |\gamma|^p$$ and there is a flat cylinder having $\gamma$ on its boundary.  
There are two cases. The first case is that the height of the cylinder is at most $\epsilon_0$. The shortest saddle connection $\beta'$ crossing the cylinder has a component in the direction of the cylinder of length at most $|\gamma|$ and an orthogonal component of length at most $\epsilon_0$. If we include $\gamma$ and $\beta'$ as part of a collection of saddle connections whose holonomy vectors (or period coordinates) define the Lebesgue measure $\mu$ we see 
$$\mu(\Omega_2(k))=O(\epsilon_0|\gamma|^3)=O(k^{-\frac{3}{2(1+\delta)}}).$$ We are using \cite[proof of Theorem 10.3]{MasurSmillie} for the proof of  the first equality. There the measure of a set of surfaces is being bounded. Since $\delta < \frac 1N$, and $N \geq 2$,  $$\frac{3}{2(1+\delta)} > 1,$$ so again $\mu(\Omega_2(k))$ is summable in $k$. The second case is the height is greater than $\epsilon_0$. 
In this case, we again apply the argument in \S\ref{sec:omega1}, observing that $P(\beta)$ cannot 
follow $\gamma$ successively more than $2M$ times because its length is less than $\epsilon_0$ 
and, by Lemma~\ref{lemma:5crossing}, our saddle connection $\beta$ would have to cross the 
cylinder that has $\gamma$ on its boundary.  
\end{proof}

\begin{Theorem}\label{theorem:main:2} 
Suppose $f$ is the characteristic function of a disc of radius $R$. Then $\widehat f\in L^2(\hh, \mu)$.

\end{Theorem}

\begin{proof}
We cover the disc of radius $R$ with sectors  of angle $\frac{\epsilon_0^2}{R^2}$.
It is enough to show that for $f$ the  characteristic function of any of these sectors the function $\widehat f$ is in $L^2(\hh,\mu)$.  
Let $\theta_0$ the center angle of this sector.  Let $t_0=\log \frac{R}{\epsilon_0}$.  Let $(X,\omega)$ any translation surface  and  consider $(Y,\omega')=g_{t_0}r_{-\theta_0}(X,\omega)$. That is, we rotate so direction $\theta_0$ is vertical and flow time $t_0$. Then since the angle is $\frac{\epsilon_0^2}{R^2}$, every saddle connection of $(X,\omega)$ 
in that sector has length at most $\epsilon_0$ on $(Y,\omega')$.  Let $h$ be the characteristic function of the disc of radius $\epsilon_0$.   Then since the flow is $\mu$ measure preserving, 
$$\int (\widehat f) ^2(X,\omega) d\mu(X)\leq\int (\widehat h)^2(Y,\omega')d\mu(Y)<\infty.$$
\end{proof}

\subsection{Proof of Theorem~\ref{theorem:main}} Let $f \in B_c(\R^2)$. then there is an $R >0$ so that the support of $f$ is contained in $B(0, R)$, and letting $C = \max f$, we have $$f \le C \chi_{B(0, R)},$$ so $$\widehat{f} \le C \widehat \chi_{B(0, R)}.$$ Applying Theorem~\ref{theorem:main:2}, we have our result.\qed
\medskip

\subsection{Optimizing exponents}\label{sec:optimizing} In fact, our proof shows the following:

\begin{Theorem}\label{theorem:optimalexponent} Let $f \in B_c(\R^2)$, Then for any $q < 1+ \sqrt{1+\frac 4 N}$, $$\widehat{f} \in L^q(\hh, \mu).$$ Here $N= 2g+|\Sigma|-1$ is the dimension of the relative homology $H_1(X, \omega, \Sigma)$ of surfaces in $\hh$. \end{Theorem}

\begin{proof} We need to show the measure $\mu_k$ of the set of surfaces $(X, \omega)$ with $$\widehat{f} \geq k^{1/q}$$  is summable in $k$.  The same argument as (\ref{eq:eskinmasur2}) shows that if $\widehat f(X,\omega)\geq   k^{1/q}$, there is a $c$ so that \begin{equation}\label{eq:eskinmasurq}|\gamma|\leq ck^{-\frac{1}{q(1+\delta)}}.\end{equation}

\noindent As before, we need to understand the set $\hh_3$. We use the partition $\hh_3 = \Omega_0 \sqcup \Omega_1 \sqcup \Omega_2$ in \S\ref{sec:decomposition}. Following \S\ref{sec:omega0}, the contribution $\mu_{k,0} $ to $\mu_k$ from $\Omega_0$ satisfies $$\mu_{k,0} <  O(|\gamma|^{2+2p}) = O\left( k^{-\frac{2(1+p)}{q(1+\delta)}}\right).$$ For this to be summable, we need $q < 2 \frac{1+p}{1+\delta}$.  Following \S\ref{sec:omega1}, the contribution $\mu_{k,1}$ from $\Omega_1$ satisfies $$\mu_{k, 1} = O(|\gamma|^2) = O\left(k^{-\frac{2}{qNp}}\right).$$ For this to be summable we need $q < \frac{2}{Np}$. Following \S\ref{sec:omega2}, the contribution $\mu_{k_2}$ from $\Omega_2$ satisfies $$\mu_{k,2} = O(\epsilon_0|\gamma|^3)=O(k^{-\frac{3}{q(1+\delta)}}).$$ For this to be summable we need $q < \frac{3}{1+\delta}$. Combining these estimates, we need $$q < \min \left( \frac{3}{1+\delta}, 2 \frac{1+p}{1+\delta}, \frac{2}{Np}\right),$$ for some $0 < \delta < p < 1/N$. To optimize this estimate, we put $\delta = 0$, and note that the first term is always the largest, so we need to compute $$\max_{0 < p < \frac 1 N} \min\left(2(1+p), \frac{2}{Np}\right) = 2\max_{0 < p < \frac 1 N} \min\left((1+p), \frac{1}{Np}\right).$$ Since $1+p$ is increasing in $p$ and $\frac {1}{Np}$ is decreasing in $p$, we set these two equal to each other, and obtain $$Np(1+p) = 1.$$ Solving this yields $$p_c = \frac{-1 + \sqrt{1+ \frac 4 N }}{2},$$ and so for any $$q < q_c = 2(1+p_c) = \frac{2}{Np_c} = 1 + \sqrt{1+ \frac 4 N },$$ we have $\widehat{f} \in L^q$, as desired. We do not know if this is truly the optimal exponent, but it is the best exponent our proof yields.\end{proof}

\subsection{Rank one orbit closures}\label{sec:rankone} Wright \cite {Wright} showed that any $(X,\omega)$ in a rank $1$ orbit closure in $\hh$ is completely periodic.  Completely periodic means that for any  direction with a  cylinder $\gamma$, the surface can be written as a union of cylinders in that direction, and there are  always such cylinder directions (actually a dense set).  Thus  the set of $(X,\omega)$ with $\widehat{f}\geq \sqrt k$ coincides with  the set $\Omega_0(k)$.  Nguyen~\cite[Proposition 4.3]{Nguyen}  proved  that for any ergodic $SL(2,\R)$ invariant measure $\nu$  on a rank $1$ orbit closure, and  any such cylinder $\gamma$,  $$\nu(\Omega_0(k))=O(|\gamma|^3).$$
Together with the discussion in the cylinder case in the proof of Theorem~\ref{theorem:main:1}, this gives $\widehat{f}\in L^2(\nu)$ for any such $\nu$ (in fact, in $L^q(\nu)$ for any $q < 3$).

\section{Siegel-Veech measures}\label{sec:spectral} In this section, we prove Theorem~\ref{theorem:spectral}, and give examples of the resulting Siegel-Veech measures in some special cases.

\subsection{Transforms and bounds} Let $\tau$ denote an $SL(2, \R)$ invariant measure on a stratum $\mathcal H$ of abelian differentials, and suppose that for any $f \in B_c(\R^2)$, $\widehat{f} \in L^2(\hh, \tau)$. Then, for any $h \in B_c(\R^4)$, $\widehat{h} \in L^1(\hh, \tau)$, since we can dominate  $$\widehat{h}(\omega) = \sum_{v_1, v_2 \in \Lambda_{\omega}} h(v_1, v_2)$$ by $(\widehat{f})^2$ where $f = \|h\|_{\infty} \chi_{H}$, where $H$ denotes the union of the projections of the support of $h$ via the maps $\R^4 = \R^2 \times \R^2 \rightarrow \R^2$ $$(x,y) \longmapsto x \mbox{ and } (x, y) \longmapsto y.$$

\subsection{Haar measures} Our claim in Theorem~\ref{theorem:spectral} is that there are \emph{Siegel-Veech measures} $\nu = \nu(\tau)$ on $\R\minuszero$ and $\eta = \eta(\tau)$ on $\mathbb P^1 (\R) =  \R \cup \infty$ such that \begin{eqnarray*} \int_{\mathcal H} \widehat{h}(\omega) d\tau(\omega) &=& \int_{\R\minuszero} \left(\int_{SL(2, \R)} h(tx, y) d\kappa(x,y) \right) d\nu(t) \\ &+ &\int_{\mathbb P^1(\R)} \left(\int_{\R^2} h(x, sx) dx\right) d\eta(s).\\ \end{eqnarray*} Here, $\kappa$ is Haar measure on $SL(2, \R)$ (with a fixed normalization), and the integral over $SL(2, \R)$ is taken over pairs $(x,y) \in \R^2 \times \R^2$ with $\det(x,y)=1$, that is, we view $SL(2, \R)$ as a subset of $\R^4 = \R^2 \times \R^2$.

This follows by the $SL(2, \R)$-invariance of $\tau$ and the classification of $SL(2, \R)$-orbits on $\R^4$. By the invariance of $\tau$, and our integrability condition, $$h \longmapsto \int_{\mathcal H} \widehat{h}(\omega)d\tau(\omega)$$ is a $SL(2, \R)$-invariant linear functional on $B_c(\R^4)$, the set of bounded compactly supported functions on $\R^4$. Therefore, there is an $SL(2, \R)$-invariant measure $m = m(\tau)$ (a \emph{Siegel-Veech measure}) on $\R^4 = \R^2 \times \R^2 = M_2(\R)$  so that $$\int_{\mathcal H} \widehat{h}(\omega)d\tau(\omega)= \int_{\R^4} h dm$$

\subsection{$SL(2,\R)$-invariant measures on $\R^4$} To describe $SL(2,\R)$-invariant measures on $\R^4$, we need to understand $SL(2, \R)$-orbits on $\R^4$. For $t \in\R$, let $$D_t = \{ (x, y) \in 
\R^2 \times \R^2: \det(x, y) = t\}$$

\noindent For $t \neq 0$, $D_t$ is an $SL(2, \R)$-orbit. $D_0$ decomposes further. For $s \in \mathbb P^1(\R)$, let  $$L_s = \{ (x, sx): x \in \R^2, x \neq 0\},$$ with $$L_{\infty} = \{ (0, y): y \in \R^2, y \neq 0\}$$

\subsubsection{Orbits and measures} $D_t$ and $L_s$ are the non-trivial $SL(2,\R)$ orbits on $\R^2 \times \R^2$, and each carries a unique (up to scaling) $SL(2, \R)$-invariant measure. These are the (non-atomic) ergodic invariant measures for $SL(2,\R)$ action on $\R^2 \times \R^2$. On $D_t$, the measure is Haar measure on $SL(2, \R)$, and on $L_s$ it is Lebesgue on $\R^2$. Thus, associated to any $SL(2,\R)$ invariant measure $m$ on $\R^4$ we have measures $\nu = \nu(m)$ and $\eta = \eta(m)$ so that \begin{eqnarray*} \int_{\R^4} h dm &=& \int_{\R\minuszero} \left(\int_{SL(2, \R)} h(tx, y) d\kappa(x,y) \right) d\nu(t) \\ &+& \int_{\mathbb P^1(\R)} \left(\int_{\R^2} h(x, sx) dx\right) d\eta(s).\\ \end{eqnarray*} 

\subsection{Siegel-Veech measures from measures on strata} Putting $\nu(\tau) = \nu(m(\tau))$ and $\eta(\tau) = \eta( m(\tau))$, we have our \emph{Siegel-Veech measures}. These measures are interesting invariants associated to $SL(2, \R)$-invariant measures $\tau$ on $\mathcal H$.  A natural question is:

\begin{question*} Let $\mu$ denote Lebesgue measure on the stratum $\hh$. What are the Siegel-Veech measures $\nu(\mu)$ and $\eta(\mu)$?
\end{question*}

\subsubsection{Virtual Triangles}\label{sec:virtual} Smillie-Weiss~\cite{SmillieWeiss} introduced the notion of \emph{virtual triangles} on a surface. A virtual triangle is simply a pair of (distinct) saddle connections, and the \emph{area} of a virtual triangle is the (absolute value of the) determinant of the matrix given by the holonomy vectors. 

They showed that there is a positive lower bound on the area of virtual triangles on the surface $\omega$ if and only if the surface $\omega$ is an \emph{lattice surface}, that is, its stabilizer $SL(X, \omega)$ under $SL(2, \R)$ is a lattice. In this case the $SL(2, \R)$ orbit is \emph{closed}, and the Haar measure $\tau$ on $SL(2,\R)/SL(X,\omega)$ is finite.

This condition, known as \emph{no small virtual triangles} (NSVT) can be summarized as saying that $\nu(\tau)$ has no support in a neighborhood of $0$. More generally, given an arbitrary $SL(2,\R)$-invariant measure $\tau$, the support of $\nu(\tau)$ is the collection of virtual triangle areas for surfaces $\omega$ in the support of $\tau$.

\subsection{Lattice surfaces and covering loci}
For some lattice surfaces and loci of covers, we have examples where we can compute these measures explicitly. 

\subsubsection{Flat tori}\label{sec:flat} For $$\mathcal H(\emptyset) = SL(2, \R)/SL(2, \Z),$$ the moduli space of abelian differentials on flat tori, these measures were implicitly computed by Schmidt~\cite{Schmidt}, see Fairchild~\cite{Fairchild} for an explicit computation with full proofs (which correct a mistake in a paper of Rogers~\cite{Rogers}). Precisely, if the Haar measure $\kappa$ on $SL(2, \R)$ is normalized so that $$\kappa(SL(2, \R)/SL(2, \Z)) = \zeta(2),$$ we have $$\nu = \sum_{n \in \Z\minuszero} \frac{1}{\zeta(2)} \phi(n) \delta_n \mbox{ and } \eta = \delta_1 + \delta_{-1}.$$ 

\subsubsection{Covering loci and affine lattices}\label{sec:cover} In the stratum $\hh(1, 1)$ we have the subvariety $\mathcal V$ of two identical tori glued along a slit. These are double covers of a flat torus branched over two points. $\mathcal V$ is a degree 4 cover of $\hh(0, 0)$, the space of tori with two distinct marked points, where the slit is a segment connecting the two points. The cover has degree $4$ since a pair of slits joining the same marked points determines the same point in $\hh(1,1)$  if the slits are in the same relative homology class mod $2$; and the relative homology group with $\Z_2$ coefficients has order $4$.  Further, $$\hh(0, 0) = A^*SL(2, R)/A^*SL(2,\Z),$$  where $$A^*SL(2, \R) = SL(2,\R) \ltimes \R^2\minuszero, A^*SL(2, \Z) = SL(2, \Z) \ltimes \Z^2\minuszero.$$ Then $(X,\omega)\in\mathcal{V}$ covers a point which we identify  as $[g, v]$, $g \in SL(2, \R)$, $0 \neq v \in \R^2/g\Z^2$.  For $(X,\omega)$ in $\mathcal{V}$ for which the holonomy vector of the slit is totally irrational (a set of full measure), we have~\cite{CheungHubertMasur} $$\La_\omega = g\Z^2_{\prim} \cup (g\Z^2 + v),$$ where $\Z^2_{\prim}$ is the set of \emph{primitive} integer vectors. Thus, we can break the Siegel-Veech measures up into the measures associated to each piece. For the first piece $g\Z^2_{\prim}$, the computation in \S\ref{sec:flat} yields the measures, and for the second, these were computed in~\cite{Athreya}.

\subsubsection{Lattice surfaces} More generally, for lattice surfaces $\omega$, it seems possible to use the fact that there are vectors $v_1, \ldots v_k \in \R^2$ so that $$\La_{\om} = \bigcup_{j=1}^k SL(X,\omega) v_j$$ to turn to algebraic techniques to compute these measures, which will depend on the action of $SL(X, \omega)$ on $\R^2 \times \R^2$. 


\section{$L^2$ bounds and error terms}\label{sec:errorbounds}

\subsection{Notation} We prove Theorem~\ref{theorem:error}. We write $$L(R) = \|N(\omega, R)\|_2^2,$$ and $$V(R) = L(R) - \|N(\omega, R)\|_1^2 = L(R) - c_{SV}(\mu)^2 \pi^2 R^4.$$

\subsection{Expectation and variance}

Then 

\begin{eqnarray*} \mu\left(\omega \in \hh:  \left|N(\omega, R) - c_{SV}(\mu) \pi R^2\right| > e(R)\right) &=& \\ \mu\left(\omega \in \hh:  \left|N(\omega, R) - c_{SV}(\mu) \pi R^2\right|^2 > e(R)^2\right) &\le& \\ \frac{1}{e(R)^2}\int_{\hh} \left|N(\omega, R) - c_{SV}(\mu)  \pi R^2\right|^2 d\mu  &=&  \frac{V(R)}{e(R)^2}.\end{eqnarray*}

\subsection{Borel-Cantelli}\label{sec:bc} Theorem~\ref{theorem:error} then follows from applying the easy part of the Borel-Cantelli lemma to the sequence of sets $$A_k = \left\{\omega \in \hh: \left|N(\omega, R) - c_{SV}(\mu) \pi R^2\right| \right\}> e(R_k).$$\qed

\subsection{Power savings on lacunary sequences}\label{sec:powersavings} We show how to combine the main result Theorem~\ref{theorem:variance} of Appendix~\ref{sec:appendix} with Theorem~\ref{theorem:error} to obtain almost everywhere power savings along a lacunary sequence of radii $R_k$ (i.e., there is a $c>1$ so that $\frac{R_{k+1}}{R_k} \geq c$. ) The results of Nevo-R\"uhr-Weiss~\cite{NRW} give this quality of bound for $R \rightarrow \infty$ in general.  Theorem~\ref{theorem:variance} yields a $\delta>0$ so that we have the bound $$V(R) = o( R^{4-\delta}).$$ Setting  $e(R_k) = R_k^{2-\eta}$, $\eta < \delta$, and applying Theorem~\ref{theorem:error}, we obtain:

\begin{Cor} Let $R_k$ be a lacunary sequence. Then for $\mu$-almost every $\omega \in \hh$, $$|N(\omega, R_k) - c_{SV}(\mu)\pi R_k^2| \le R_k^{2-\eta}.$$

\end{Cor}

\newpage 
\appendix

\section{Variance Estimates \\ \smallskip \small Jayadev S.~ Athreya and Rene R\"uhr}\label{sec:appendix}

\subsection*{Notation} $\hh$ is a connected component of a stratum of area one translation surfaces,  and $N(\omega, R)$ is function $\widehat{\id_{B_R}}(\omega)$ where $B_R$ is the Euclidean disc centered at zero of radius $R$ in $\mathbb C$. The goal of this appendix is to prove the following asymptotic estimate for the $L^2$ norm of the counting function $N(\cdot, R)$. We fix some notation. For $g \in SL(2, \R)$, and a function $f$ on a space $Y$ on which $SL(2, \R)$ acts, we write $$g\acts f(y) = f(g^{-1}y).$$ Our spaces will be either $Y= \R^2$ where $SL(2, \R)$ acts linearly, or $Y = \cH$. We write $c_{\mu}$ to be the Siegel-Veech constant for the Lebesgue probability measure $\mu$ on $\cH$, so that for any bounded compactly supported function $f$ on $\R^2$, the Siegel-Veech formula~\cite{VeechSiegel} yields \begin{equation}\label{eq:sv}\int_{\cH} \widehat{f} d\mu = c_{\mu} \int_{\bR^2}f(x)dx.\end{equation}

\begin{Theorem}\label{theorem:variance} There exists $\delta>0$ such that for any sufficiently large $R$,
\[
\|N(\cdot,R)\|_{L^2(\cH,\mu)}=c_\mu\pi R^2+\cO(R^{2-\delta}).
\]
\end{Theorem}
\begin{proof}
The key idea in the proof comes from the work of Eskin-Masur~\cite{EskinMasur}: There is a function $\phi \in B_c(\bR^2)$ with $\int_{\bR^2} \phi(x) dx=1$ such that
$$\left(A_R \widehat{\phi}\right) (\omega) := \frac{1}{2\pi} \int_{0}^{2\pi} \widehat{\phi}(a_{R} r_{\theta} \omega)d\theta \approx\frac{1}{\pi R^2}N(\omega, R)$$
where $a_R=\diag{R,R^{-1}}$ and $r_{\theta}$ is counterclockwise rotation by $\theta$. On the other hand,
\[
\|A_R f\|_{L^2(\cH,\mu)}^2=\langle a_R \acts A_R f,f\rangle_{L^2(\cH,\mu)}
\]
and the right hand side can be approximated with effective bounds on matrix coefficients of $L^2(\cH,\mu)$ (effective mixing of the $a_R$-action) so that $$\|A_R f\|_{L^2(\cH,\mu)}^2\approx \left(\int_{\cH} f d\mu\right)^2.$$
(\ref{eq:sv}) gives  $$\int_{\cH}\widehat{\phi}d\mu = c_\mu \int_{\bR^2} \phi(x) dx = c_{\mu},$$ so we deduce $$\|N(\omega, R)\|_{L^2(\cH,\mu)}\approx c_\mu \pi R^2.$$

\noindent\textbf{Sectors.} We now make those estimates precise, recalling the argument of Nevo-R\"uhr-Weiss~\cite{NRW}, which in turn is based on ideas of Eskin-Margulis-Mozes~\cite{EMM}, Veech~\cite{VeechSiegel} and Eskin-Masur~\cite{EskinMasur}.
Start by taking an $\epsilon$-approximation $\phi_\eps$ of the characteristic function of a sector of the unit ball symmetric at the $y$-axis with total angle $2\theta$. 
The key observation is that $a_R^{-1}\acts\phi_\eps$ approximates now a sector of the ball of radius $R$ with new angle $$\theta_R=\theta R^{-2}+O(\theta^3 R^{-2}).$$ 
Computing $A_R \phi_\eps (x)$ (i.e., integrating $\phi_\eps(a_R r_{\theta} x)$ in $\theta$ from $0$ to $2\pi$) yields $\theta_R/{\pi}$
if and only if $x$ has length less than $R$, and is 0 otherwise. 
There is a small approximation problem: Dilating a sector by $a_R$ does not give a sector again. 
However, $a_R$ does map triangles to triangles, 
and we approximate the original sector $S$ by two such triangles $W_1\subset S\subset W_2$ with $W_1\subset B_1$ of angle $2\theta$, 
touching at the corners with $\bS^1$, and thus is of height $\cos\theta$ 
and the midpoint of the top of $W_2$ touching $\bS^1$. 
There are smooth bump functions $0\leq \phi_{\pm,\eps}\leq 1$ (see \cite[p12]{NRW}) that approximate $W_1$ from its interior and $W_2$ from the outside respectively, 
and satisfy $$\int_{\R^2} \phi_{\pm,\eps}(x)dx=R^2\theta_R+O(\eps^{1/2}R^2\theta_R),$$ where  $$\|\partial_\theta \phi_{\pm,\eps}\|_{L^\infty(\bR^2)} = O(\eps^{-1}).$$ 
This discussion is summarized in the following chain of inequalities:
\begin{eqnarray*}
A_R\widehat{\phi_{-,\eps}}(\omega) &\leq& A_R\widehat{\mathbbm{1}_{W_1}}(\omega) \leq \frac{\theta_R}{\pi}N(\omega, R)\\
&\leq& A_R\widehat{\mathbbm{1}_{W_2}}(\omega)\leq A_R\widehat{\phi_{+,\eps}}(\omega).
\end{eqnarray*}

\noindent\textbf{Effective equidistribution.} We need two definitions before stating our key lemma. We say $f \in L^2(\cH, \mu)$ is $K$-\emph{finite} if the span of the set of functions $\{r_{\theta}\acts f: 0 \le \theta < 2\pi\}$ is finite-dimensional. We define the $K$-\emph{Sobolev norm} $\cS_{K}(f)$ of such a function by $\cS_K^2(f)=\|f\|_{L^2(\cH,\mu)}^2+\|\eta f\|_{L^2(\cH,\mu)}^2$ on $L^2(\cH,\mu)$ where $\eta$ is a generator of $\Lie(\bS^1)$, i.e., $\eta$ is a differential operator coming from the action of $r_{\theta}$. We say $f$ is $K$-\emph{smooth} if this norm is finite.

\begin{lemma}
\label{operatorbound}
There exists $\kappa>0$ such that for any $K$-smooth $f\in L^2(\cH,\mu)$ with $\|f\|_\infty\le \cS_K(f)$,
\begin{eqnarray*}
\left(\|A_R f\|_{L^2(\cH,\mu)}^2-\left(\int_{\cH} fd\mu\right)^2\right)^{1/2} &=& \left\|A_R f-\int_{\cH} f d\mu \right \|_{L^2(\cH,\mu)} \\ &\ll& R^{-\kappa}\cS_K(f).
\end{eqnarray*}
\end{lemma}
\begin{proof}
This is Theorem 3.3 in \cite{NRW}. We only recall here that it follows from \cite{AGY}, by which there exists $\kappa>0$ and $C>0$ such that for any $g\in\SLR[2]$ and any $f_1,f_2\in L^2(\cH,\mu)$ that are not constant $K$-eigenfunctions 
\[
|\langle g\cdot f_1,f_2\rangle_{L^2(\cH,\mu)}|\leq C\|f_1\|_{L^2(\cH,\mu)}\|f_2\|_{L^2(\cH,\mu)}\|g\|^{-\kappa}.
\]
By a Fourier decomposition argument, one obtains the claimed inequality for functions with finite $K$-Sobolev norm.
\end{proof}

\noindent\textbf{Completing the proof.} To complete the proof of the theorem, we'll pick $\theta$ to depend on $R$ and $\eps$ to depend on $\theta$ to achieve our desired bound. Note that by equivariance of the Siegel-Veech transform, $$\phi_{\pm,\eps}\leq \id_{B_1} \mbox{ and }\|\partial_\theta \phi_{\pm,\eps}\|_{L^\infty(\bR^2)} = O(\eps^{-1}).$$ We deduce $\cS_K(\widehat{\phi_{+,\eps}}) = O( \eps^{-1})$ (see \cite[Lemma 3.4]{magee2018counting}). Set $\eps = \theta^2$, so $\eps^{1/2}=\theta$. 
Then
\begin{eqnarray*}
\|N(\cdot, R)\|_{L^2(\cH,\mu)}^2 &\leq& \frac{\pi^2}{\theta_R^2}\|A_R\widehat{\phi_{+,\eps}}\|_{L^2(\cH,\mu)}^2 \\ 
&=& \frac{\pi^2}{\theta_R^2}\left(\left(\int_{\cH}\widehat{\phi_{+,\eps}}d\mu\right)^2+O(R^{-2\kappa}\cS_K(\widehat{\phi_{+,\eps}})^2)\right)\\
&=& \frac{\pi^2}{\theta_R^2}\left((c_\mu R^2\theta_R+O(\theta R^2\theta_R))^2+O(R^{-2\kappa}\theta^{-4})\right)\\
&=& \pi^2R^4\left(c_\mu^2+O(\theta)\right)+O(R^{4-2\kappa}\theta^{-6}).\\
\end{eqnarray*}

\noindent The lower bound follows verbatim. Choose $\delta>0$ and put $\theta=R^{-\delta}$ to match error terms and to finish the proof of the theorem.
\end{proof}

\end{document}